\documentclass{article} 
\usepackage{amsthm,amsmath,amsfonts,amssymb, hyperref,cases,xcolor,tikz,authblk}
\usepackage[capitalise]{cleveref}
\usepackage{vmargin}
\setmarginsrb{2cm}{1cm}{2cm}{3cm}{1cm}{1cm}{2cm}{2cm}

\numberwithin{equation}{section}
\newtheorem{Theorem}{Theorem}[section]
\newtheorem{Lemma}[Theorem]{Lemma}
\newtheorem{Proposition}[Theorem]{Proposition}

\renewcommand{\leq}{\leqslant}

\renewcommand{\geq}{\geqslant}

\newcommand{\dd}{\textrm{d}}

\renewcommand{\div}{\operatorname{div}}
\title{Existence of contacts for the motion of a rigid body into a viscous incompressible fluid with the Tresca boundary conditions}    
\author[1]{Matthieu Hillairet}
\author[2]{Tak\'eo Takahashi}
\affil[1]{Institut Montpellierain Alexander Grothendieck, CNRS, Univ Montpellier}  \affil[2]{Universit\'e de Lorraine, CNRS, Inria, IECL, F-54000 Nancy, France}

\date{\today}

\begin{document}                  

\maketitle

\abstract{
We consider a fluid-structure interaction system composed by a rigid ball immersed into a viscous incompressible fluid. The motion of the structure satisfies the Newton laws and the fluid equations are the standard Navier-Stokes system. At the boundary of the fluid domain, we use the Tresca boundary conditions, that permit the fluid to slip tangentially on the boundary under some conditions on the stress tensor. More precisely, there is a threshold determining if the fluid can slip or not and 
there is a friction force acting on the part where the fluid can slip. 
Our main result is the existence of contact in finite time between the ball and the exterior boundary of the fluid for this system in the bidimensional case and in presence of gravity.
}

\vspace{1cm}



\tableofcontents

%%%%%%%%%%%%%%%%%%%%%%%%%%%%%%%%%%%%%%%%%%%%%%%%%%%
%%%%%%%%%%%%%%%%%%%%%%%%%%%%%%%%%%%%%%%%%%%%%%%
%%%%%%%%%%%%%%%%%%%%%%%%%%%%%%%%%%%%%%%%%%%%%%%
\section{Introduction}
The system composed by a rigid body and a viscous incompressible fluid, assuming no-slip of the fluid on the solid boundaries, has been studied thoroughly from a mathematical point of view \cite{Serre, Judakov, MR1759801, MR1682663, MR1763528, MR1781915}. These results yield that we have similar well-posedness properties as for the fluid alone prior to a possible contact between the rigid body and the exterior boundary. 
One important issue is then to understand what happens at the time of the contacts if they exist. In \cite{MR1870954}, the authors show that in dimension 2, if there is a contact between solids, it occurs with null relative velocity and acceleration. 
% This result was extended in \cite{MR2019028} for the dimension 3 in space.
Then, in \cite{MR2354496, MR2481302} it is proved that for some particular geometry, in dimension 2 or 3 in space,  no contact occurs in finite time. For deformable structures, a similar result is proved in \cite{MR3466847}.
All these results are obtained again under the assumption that the fluid does not slip on the solid boundaries. 
One of the remedies to recover contacts is to take into account that, in presence of high shear on the boundaries, the fluid should be allowed to slip.  One classical model which includes this phenomenon is the Navier slip boundary conditions  \cite{Navier}. For these boundary conditions, the case of an immersed rigid body is studied in \cite{GVH, GVHW, Wang2014}. The authors obtain the well-posedness of the corresponding system up to contact and show the existence of contact in finite time between rigid bodies (in dimension 2 and 3). We mention also that the Cauchy theory for a model including slip on the moving body but no-slip
on the container boundary is studied in \cite{Sarka}.

One drawback of the Navier slip boundary conditions is that it forces the fluid to slip tangentially whatever the size of the shear on the boundaries. A more realistic model are the Tresca boundary conditions.  In these boundary conditions, the fluid sticks to the interface up to a shear-rate threshold that the fluid is prevented to exceed by allowing  slip on the interface. The boundaries of the fluid domain split then in a zone of small shear rates where Dirichlet boundary conditions are imposed and high shear rates where a type of Navier boundary conditions are imposed (but with an unknown slip length which encodes that the shear rate cannot exceed the threshold value). Discussing whether such models allow contacts or not is a delicate issue. Indeed, the intuitive idea would be to throw sufficiently fast the body toward the container boundary.  This would create high shear rate and induce slip on the solid boundaries which do not prevent from contact (again see \cite{GVHW}). 
Yet, though less singular than the Dirichlet boundary conditions, the Navier  boundary conditions also imply a kinetic-energy dissipation that forces the velocity of the moving body to vanish when contact occurs. Hence, one must be careful that this dissipation is not sufficiently fast to decrease the shear rate in a sufficiently large zone below the disk implying that no-slip boundary conditions appear preventing from collision occurence (see \cite{MR2354496}).  To discuss this issue, we focus on a simplified 2D symmetric configuration similar to \cite{GVHW,MR2354496,MR2481302}. We focus on the 2D case since the knowledge on the Dirichlet problem shows that contacts occur
in the 2D case with more difficulties than in the 3D case.  

%%%%%%
%%%%%
%%%%%

\medskip

%%%%%%
%%%%%
%%%%%

We describe now the system under consideration in this paper. 
We assume that the rigid body is a ball of radius 1 and that the container $\Omega$ is a rectangle:
\begin{equation*}%\label{geometry}
B_h = B((h+1) e_2,1), \quad \Omega= (-L,L) \times (0,L'),
\end{equation*}
with $(e_1,e_2)$ the canonical basis of $\mathbb R^2,$ and
with $h>0$ the distance between the rigid body and the container boundary (it will be a function of time in what follows) and $L>1$, $L'>2$ two constants.
% We recall the geometry in Figure \ref{F1}.
%%%%%%%%%%%%%%%%%%%%%%%%%%%%%%%%%%%%%%%
\begin{figure}%\label{F1}
\begin{center}
\begin{tikzpicture}
\draw (0,0) rectangle (8,6);
\draw[dashed] (4,-1) -- (4,7);
\draw (4,1.5) circle(1);
\draw (4,1.5) node {$\times$};
\draw (2,3) node {$\Omega_h$};
\draw[densely dotted,<->] (3.9,0) --  (3.9,0.5); 
\draw (3.9,0.25) node[left] {$h$};
\draw (0,0) node[below] {$-L$};
\draw (8,0) node[below] {$L$};
\draw (4,0) node[below right] {$0$};
\draw (4,6) node[above right] {$L'$};
\end{tikzpicture}
\end{center}
\caption{Geometry and notations}
\end{figure}
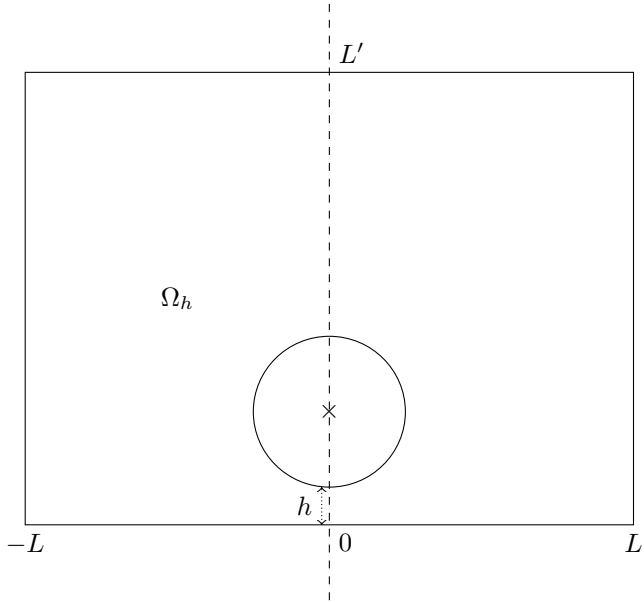
%%%%%%%%%%%%%%%%%%%%%%%%%%%%%%%%%%%%%
Corresponding to the position $B_h$ of the rigid body, we denote by 
$$
\Omega_h = \Omega \setminus \overline{B_h}
$$ 
the fluid domain. The equations that govern our fluid-solid system write
\begin{equation}\label{tre5.0}
\left\{
\begin{array}{rcl}
\rho_{\mathcal{F}}(\partial_t u + u \cdot \nabla u ) -\Delta u+\nabla p&=& 0 \quad \text{in}\  \Omega_h,\\
\div u&=&0 \quad \text{in}\  \Omega_h,\\
\end{array}
\right.
\end{equation}
\begin{equation}\label{tre5.1}
u\cdot n = \lambda e_2\cdot n \quad \text{on} \quad \partial B_h,
\end{equation}
\begin{equation}%\label{tre5.3}
\begin{cases}
u\cdot \tau -\lambda e_2\cdot \tau = 0 &\text{if} \quad |D(u)n\cdot \tau|<1,\\
\exists \beta \geq 0\quad u\cdot \tau -\lambda e_2\cdot \tau = -\beta D(u)n\cdot \tau & \text{if} \quad |D(u)n\cdot \tau|=1,
\end{cases}
\quad \text{on} \quad \partial B_h,
\end{equation}
\begin{equation}%\label{tre5.4}
u\cdot n = 0 \quad \text{on} \quad \partial \Omega,
\end{equation}
\begin{equation}\label{tre5.5}
\begin{cases}
u\cdot \tau = 0 &\text{if} \quad |D(u)n\cdot \tau|<1,\\
\exists \beta \geq 0\quad  u\cdot \tau = -\beta D(u)n\cdot \tau & \text{if} \quad |D(u)n\cdot \tau|=1,
\end{cases}
\quad \text{on} \quad \partial \Omega,
\end{equation}
\begin{equation}%\label{tre5.2bis}
\dot{h}=\lambda,
\end{equation}
\begin{equation}\label{tre5.2}
m\dot \lambda =-\int_{\partial B_h} \Sigma(u,p)n\cdot e_2\ d\gamma -m_a g,
\end{equation}
In the above equations $n$ and $\tau$ are the fluid exterior normal and associated tangential unitary vectors, 
$$
D(u):=\frac12 \left(\nabla u+(\nabla u)^\top\right), \quad \Sigma(u,p):=2D(u)-p I_2,
$$
$m$ is the mass of the rigid ball. We assume that the rigid body is homogeneous so that $m=\pi \rho_{\mathcal{S}}$ where $\rho_{\mathcal{S}}>0$ is the constant density of the structure. We also assume that the density of the fluid $\rho_{\mathcal{F}}>0$ is a positive constant.
The constant $m_a$ is equal to $m-\pi \rho_{\mathcal{F}}=\pi (\rho_{\mathcal{S}}-\rho_{\mathcal{F}})$ and we assume $\rho_{\mathcal{S}}>\rho_{\mathcal{F}}$ (so that the ball is falling). We should point out that, since we consider a symmetric configuration, we have that the ball does not rotate which allows us to remove the conservation of linear angular momentum. To simplify, we take the viscosity of the fluid constant and equal to 1. We have also fixed the shear threshold to be equal to 1.
Both simplifications are independent and do not restrict the generality. 
We complement the system with initial data:
\begin{equation}\label{ci}
h(0)=h^0, \quad \lambda(0)=\lambda^0, \quad u(0,\cdot)=u^0 \quad \text{in} \ \Omega_{h^0}.
\end{equation}
The existence of weak solutions is tackled in dimension 3 in space in \cite{BSMT17}. We explain now how to adapt this definition to our framework.
For this, we set first: 
\begin{equation}%\label{tre2.1}
w^*=w^*_h=\begin{cases}
e_2 & \text{on} \ \partial B_h\\
0 & \text{on} \ \partial \Omega.
\end{cases}
\end{equation}
Then the Tresca boundary conditions \eqref{tre5.1}--\eqref{tre5.5} write
\begin{equation}\label{tre2.2}
(u-\lambda w^*)\cdot n = 0 \quad \text{on} \quad \partial \Omega_h,
\end{equation}
\begin{equation}%\label{tre2.3}
\begin{cases}
(u-\lambda w^*)\cdot \tau = 0 &\text{if} \quad |D(u)n\cdot \tau|<1,\\
\exists \beta \geq 0\quad  (u-\lambda w^*)\cdot \tau  = -\beta D(u)n\cdot \tau & \text{if} \quad |D(u)n\cdot \tau|=1,
\end{cases}
\quad \text{on} \quad \partial \Omega_h.
\end{equation}
We recall that this Tresca boundary conditions admits the following variational formulation \cite{BSMT17}:
\begin{equation}\label{tre0.6}
\forall c\in \mathbb{R}, \quad (D(u)n\cdot \tau) c \geq |(u-\lambda w^*)\cdot \tau|- |(u-\lambda w^*)\cdot \tau+c| \quad \text{on} \quad \partial \Omega_h.
\end{equation}
%We set also
%$$
%\rho=\begin{cases}
%\rho_{\mathcal{F}} & \text{in} \ \Omega_h\\
%\rho_{\mathcal{S}} & \text{in} \ B_h
%\end{cases},
%\quad
%\rho^0=\begin{cases}
%\rho_{\mathcal{F}} & \text{in} \ \Omega_{h^0}\\
%\rho_{\mathcal{S}} & \text{in} \ B_{h^0}
%\end{cases},
%$$
%and 
We extend also $u$ by $\lambda w^*$ in $B_{h}$, so that $\div u =0$ in $\Omega$
and $u\cdot n = 0$ on $\partial \Omega$
(thanks to \eqref{tre2.2}). If $u \in L^2(\Omega),$ these properties are summarized by the statement $u \in L^2_{\sigma}(\Omega).$ 
With these conventions, we say that $(u,h)$ is a weak solution of \eqref{tre5.0}--\eqref{ci} on $(0,T)$ if
\begin{gather}
h\in W^{1,\infty}(0,T), \quad 0<h<L'-2,\quad \dot h=\lambda
\\
u\in L^\infty(0,T;L^2_\sigma(\Omega)), \quad u=\lambda e_2 \ \text{in} \ B_h, \quad u_{|\Omega_h} \in L^2(0,T;H^1(\Omega_h)),
\end{gather}
and if, for any $(w,\ell)$ satisfying 
\begin{equation}%\label{new0.0}
w \in C^1([0,T]; L^2_\sigma(\Omega)), \quad w=\ell e_2 \ \text{in} \  B_h, \quad 
\ell\in C^1([0,T]), \quad w_{|\Omega_h} \in L^\infty(0,T;H^1(\Omega_h)),
\end{equation} 
there holds:
%{\color{red}
\begin{multline}\label{tre6.1}
\dfrac{\dd}{\dd t} \left[\int_{\Omega_h} \rho_{\mathcal{F}} u \cdot w \ dx + m \lambda \ell \right] 
- \int_{\Omega_h} \rho_{\mathcal{F}} u \cdot  (\partial_t w + (u \cdot \nabla) w  ) \ dx-m \lambda \dot \ell
+ \int_{\Omega_h} 2 D(u) : D(w) \ dx 
\\
+ m_a g \ell
+ 2\int_{\partial \Omega_h} |(u-\lambda w^*)\cdot \tau|  -  |[u-w-(\lambda-\ell) w^*]\cdot \tau|  \ d\gamma \leq 0.
\end{multline}
%}
This last identity is obtained by multiplying \eqref{tre5.0} by $w,$ integrating by parts,
introducing \eqref{tre5.2} and reformulating boundary terms thanks to \eqref{tre0.6}. 
With similar arguments as in \cite{BSMT17}, one obtains  the following result for our system:
\begin{Theorem} \label{thm_old}
Assume  $(h^0,\lambda^0) \in (0,L'-2) \times \mathbb R$ and $u^0 \in L^2(\Omega_{h^0})$ satisfy:
\[
{\rm div} u^0 = 0, \qquad u^0 \cdot n  = \lambda^0 w^* \cdot  n \quad \text{on $\partial \Omega_{h^0}$}.
\] 
Then, there exist $T>0$ and a weak solution $(u,h)$ of \eqref{tre5.0}--\eqref{ci} on $(0,T)$
that satisfies:
\begin{multline}\label{nrj}
\sup_{(0,T)} \left[ \dfrac{1}{2} \int_{\Omega} \rho |u|^2 \ dx  + m_a g h\right] 
+ 2 \int_{0}^T \int_{\Omega_{h(t)}}  |D(u)|^2 \ dx\ dt
+2 \int_0^T \int_{\partial \Omega} |u -  \lambda w^*| \ d\gamma \ dt
 \\
 \leq 
\dfrac{1}{2}\int_{\Omega} \rho^0 |u^0|^2 \ dx + m_ag h^0.
\end{multline}
Furthermore we have the following alternative:
\begin{itemize}
\item $T= \infty$
\item $T<\infty$ and $\Big(\lim_{t \to T} h(t) =0$ or $\lim_{t \to T} h(t) =L'-2\Big)$
\end{itemize}
\end{Theorem}
We omit the proof for conciseness. The main objective of this paper is to prove that the second alternative can occur, meaning that the rigid ball can touch the exterior boundary in finite time.
To this end, we are going to assume that the initial velocity of the body vanishes.
We fix $u^0 \in {L^2(\Omega_{h^0})}$ and $\rho_{\mathcal{F}}$ and take $m$ large enough and $h^0$ small enough. 
Our main result reads then:
\begin{Theorem}\label{T02}
Given $\lambda^0 = 0$ and  $u^0\in L^2(\Omega_{h^0})$. 
For $h^0$ small enough and $mh^0$ large enough, there exist $T\in (0,\infty)$ and a weak solution $(u,h)$ such that 
$$
\lim_{t \to T} h(t) =0.
$$
\end{Theorem}

The proof of \cref{T02} is based on analysing the properties of the solution inherited from Theorem \ref{thm_old}. To this end, we construct a particular family of test functions in the same spirit as in \cite{MR2354496, MR2481302, GVHW, Wang2014}. We use this family in our variational inequality and this leads us to a differential inequality on $h$ where appears a term of order $\ln(h)$ due to our boundary conditions. Using this differential inequality and taking $m$ large enough and $h^0$ small enough, we show \cref{T02}.

\medskip

The outline of the paper is as follows. In the next section, we present the most novel arguments of our analysis: an adapted Korn inequality and the treatment of the differential inequality 
leading to contact. In the last section, we introduce the weak formulation 
of our problem, discuss the construction of a particular family of test functions and we use these test functions in our weak formulation to deduce the expected differential inequality on $h.$ 

\section{Main steps in the proof of \cref{T02}}

We provide in this section two major steps in the proof of Theorem \ref{T02}. The first one consists of a Korn inequality. We recall that such inequalities are introduced to control the $L^p$-norm of a full gradient by the same $L^p$-norm of the symmetric part of this gradient. 
Such inequalities are classical but, in our case, we are specifically interested in the dependence of the constant appearing in this inequality on geometrical parameters (especially the distance $h$). 
So, we give here a detailed analysis of this point. Moreover, it turns out that we control a supplementary term  which helps a lot the analysis. The second part of this section is devoted to the final step in the proof of Theorem \ref{T02}. Thanks to a multiplier argument we obtain in 
next section a differential inequality 
for the distance $h.$ We show in this section that this differential inequality yields to finite-time contact. 

\subsection{An adapted Korn inequality}

In the whole subsection $h >0$ and we consider functional inequalities in the associated
domains $\Omega_h$ and $B_h.$ We restrict to values of $h$ lower than some fixed $h^0 \in (0,L'-2)$ since we want to consider possible contacts with the bottom boundary of 
$\Omega$ only.  Given $h \in (0,L'-2),$ we set:
$$
V^1(h):=\left\{u\in L^2(\Omega) \ ; \ u\in H^1(\Omega_h), \quad \div u =0 \quad \text{in} \ \Omega, \quad u\cdot n=0 \quad \text{on} \ \partial \Omega, \quad \nabla u=0 \quad \text{in} \ B_h\right\}.
$$
We note that, given $u\in V^1(h)$ there exists a unique $\lambda_u \in \mathbb{R}^2$ such that $u=\lambda_u$ in $B_h$. Since $u$ is globally divergence-free on $\Omega,$ we have also that 
\[
u\cdot n = \lambda_u\cdot n \quad \text{on} \ \partial B_h.
\]
We start with estimating $\lambda_u$: 

\begin{Lemma}\label{L02}
Assume $h \in (0,h^0)$ and $u \in V^1(\Omega_h)$. Then $\lambda=\lambda_u$ satisfies
\[
|\lambda|^2 \leq C_1 \int_{\Omega_h} |D(u)|^2 \ dx.
\]
for a constant $C_1$ depending only on $L,L',h^0.$ 
\end{Lemma}
\begin{proof}
We set 
$$
\Omega_h^+=\left\{x\in \Omega_h \ ; \ x_2>1+h\right\}.
$$
This is a locally Lipschitz domain as long as $0<h<L'-2$. Assume $u\in V^1(h)$, we write 
$$
0=\int_{\Omega_h^+} \div u \ dx = \int_{\partial \Omega_h^+} u\cdot n \ ds
=-\int_{\Gamma_1} u_2  \ ds+\lambda \cdot \int_{\Gamma_2}  n \ ds,
$$
where $\Gamma_1=\{(x_1,1+h) \ ; \ x_1 \in (-L,-1)\cup (1,L)\}$ and $\Gamma_2=\partial B_h\cap \partial \Omega_h^+$.
We deduce from the above relation that
$$
\lambda_2 = -\frac{1}{2} \int_{(-L,-1)\cup (1,L)} u_2(x_1,1+h)  \ dx_1.
$$
Using that for almost every $x_1\in (-L,-1)\cup (1,L)$, $u_2(x_1,\cdot)\in H^1((1+h,L'))$ and $u_2(x_1,L')=0$, we deduce that
$$
\left| \lambda_2 \right| \leq  \frac{1}{2} \int_{\Omega_h} \left|\partial_{2} u_2(x)\right| \ dx\leq C\|D(u)\|_{L^2(\Omega_h)}.
$$
We can proceed similarly for $\lambda_1$ by integrating the divergence-free condition on $\Omega_h^{r} = \{ x \in \Omega_h ; x_1 >0\}$
and we deduce the result.
\end{proof}

We can now prove the following result:
\begin{Lemma}%\label{LA2}
Assume $h \in (0,h^0)$ and $u \in V^1(h)$. Then there exists a constant $C$ depending only on $L,L',h^0$ such that
\[
\int_{\Omega_h} |\nabla u|^2 \ dx+ \int_{\partial B_h} |u|^2 \ ds\leq C \int_{\Omega_h} |D(u)|^2 \ dx.
\]
\end{Lemma}
\begin{proof}
Using a density argument, we can assume that $u \in C^{\infty}(\overline{\Omega}_h)$.
First, we have
$$
2 \int_{\Omega_h} |D(u)|^2 \ dx  = \int_{\Omega_h} (\nabla u + (\nabla u)^{\top} ):  \nabla u  \ dx.
$$
By integration by parts and using that $\div u=0$, we deduce
$$
2 \int_{\Omega_h} |D(u)|^2 \ dx  = \int_{\Omega_h} |\nabla u^2| \ dx 
+ \int_{\partial \Omega} [(u \cdot \nabla ) u ] \cdot n \ ds 
+ \int_{\partial B_h} [(u \cdot \nabla ) u ] \cdot n \ ds,
$$
where we recall that $n$ is the unit outer normal to $\Omega_h$.
Using that $u\cdot n=0$ on $\partial \Omega$ and that $\Omega$ is a rectangle, we deduce 
$$
\int_{\partial \Omega} [(u \cdot \nabla ) u ] \cdot n \ ds =0.
$$

\medskip

Finally, we compute the boundary integral on $\partial B_h.$
 For this, we introduce $(r,\theta)$ the cylindrical coordinates centered in the center of $B_h$ and $(e_r,e_{\theta})$ the associated local basis. We have then 
 $n = -e_r$ and, writing differential operators in terms of $(r,\theta)$ coordinates,
 we obtain: 
 $$
(u\cdot\nabla)u \cdot e_r = u_r \partial_r u_r + \frac{1}{r} u_\theta \partial_\theta u_r - \frac{1}{r} u_\theta^2
 $$
 and
 $$
\div u= \partial_{r} u_r + \dfrac{u_r}{r} + \dfrac{\partial_{\theta} u_{\theta}}{r} = 0. 
 $$
 This yields 
  \[
 \int_{\partial B_h} [(u \cdot \nabla ) u ] \cdot n \ ds
  = \int_{\partial B_h} |u|^2  \ ds-2  \int_{0}^{2\pi} u_\theta \partial_\theta u_r \ d\theta.
 \]
Here, we may apply that $u_{r} = - \lambda \cdot e_r $ on $\partial B_h$ so that
we can dominate:
\[
\left| 2  \int_{0}^{2\pi} u_\theta \partial_\theta u_r \ d\theta \right|
\leq C |\lambda|^2 + \dfrac{1}{2} \int_{\partial B_h} |u|^2  \ ds.
\]
Finally, we obtain that
\[
2 \int_{\Omega_h} |D(u)|^2 \ dx \geq \int_{\Omega_h} |\nabla u|^2 \ dx
+ \dfrac{1}{2}  \int_{\partial B_h} |u|^2 \ ds - C |\lambda|^2. 
\]
It remains to apply the previous lemma to control the $\lambda$ term.
This ends the proof.
 \end{proof}

Using Poincar\'e's inequality on $u_1$ and on $u_2$ and using that 
$u_1$ vanishes on the left/right boundaries (resp. $u_2$ vanishes on the top/bottom boundaries), we obtain that 
\[
\int_{\Omega_h} |u|^2 \ dx \leq C \int_{\Omega_h} |D(u)|^2 \ dx,
\]
for a constant $C$ depending only on $L,L'$. 
As a consequence, we deduce the following result:
\begin{Proposition}\label{propkorn}
Assume $h \in (0,h^0)$ and $u \in V^1(h)$, there exists a constant $C$ depending only on $L,L',h^0$ such that
\[
\int_{\Omega_h} |\nabla u|^2 \ dx+\int_{\Omega_h} |u|^2 \ dx+ \int_{\partial B_h} |u|^2 \ ds\leq C \int_{\Omega_h} |D(u)|^2 \ dx.
\]
\end{Proposition}

To end up this section, we note that this proposition also implies another control of 
$\lambda$.  This is the content of the next lemma:
\begin{Lemma} \label{etcacontinueencoreetencore}
Assume $h\in (0,h_0)$ and $u \in V^1(h)$, with $u=\lambda e_2$ in $B_h$.Then 
\[
|\lambda | \leq C h^{1/4} \|D(u)\|_{L^2(\Omega_h)}.
\]
with a universal constant $C.$
\end{Lemma}
For a proof, we refer to the arguments leading to Lemma 6.4 in \cite{Wang2014} (see inequality (A.15)). We note here that in the proof of \cite{Wang2014}, we can use the $L^2$-norm on $\partial B_h$ instead of the $L^2$-norm on $\partial \Omega$.
Then, we apply the above Korn inequality to conclude.

\subsection{Final step of the proof of \cref{T02}}
Following the assumption of \cref{T02}, we consider in this subsection that $\lambda^0=0$ and that $\rho_{\mathcal F}$ and $u_0 \in L^2(\mathcal F^0)$
are given. We consider then the weak solution $(u,h)$ provided by \cref{thm_old} for this initial data.   
This solution is defined on $(0,T)$ where $T$ is finite if contact occurs in time $T.$
We shall prove that, under the assumption that $h^0$ is sufficiently small and $mh^0$
is sufficiently large, then $T$ might not exceed some value $T_*.$

\medskip

The proof of \cref{T02} relies on two ingredients. 
The first one is a differential inequality derived in the next section. We summarize this result in the following lemma:
\begin{Lemma} \label{lem_ode}
There exist $C^{\sharp} >0$ and $C^* >0$  independent of $(m,h_0)$ such that as long as the distance function $h$ satisfies $h \leq 1,$ there holds: 
\begin{equation}\label{ode01b}
\dot h(t) \leq  - \dfrac{gt}{2}+ C^\sharp gh^0+ \frac{C^*}{m}\int_0^t |\ln(h(s))| \ d s\,.
\end{equation} 
\end{Lemma}
The proof of this lemma is postponed to the next section. The second ingredient of the proof is the energy inequality \eqref{nrj}. 
In case $\lambda^0=0$ the initial energy of the system reads:
\begin{equation}\label{tre4.2}
E^0:=\dfrac{1}{2}\int_{\Omega} \rho^0 |u^0|^2 \ dx + m_ag h^0=\dfrac{\rho_{\mathcal{F}}}{2}\int_{\Omega_{h^0}} |u^0|^2 \ dx + (m-\pi \rho_{\mathcal{F}}) g h^0.
\end{equation}
So, taking $mh^0$ sufficiently large, we have $E^0 \leq2 mgh^0$ 
so that energy estimate entails:
\begin{equation} \label{ingredient1b}
|\dot{h}(t)| \leq 2\sqrt{gh^0}, \qquad \forall \, t \in (0,T).
\end{equation}

From now on, we suppose \eqref{ingredient1b} and \eqref{ode01b}. 
We fix $\sigma \in (0,1/2)$ and we choose $h^0$ sufficiently small and $mh^0$
sufficiently large so that:
\begin{equation}
\label{ode17bis}
\left\{
\begin{aligned}
h^0 & < \max \left( \frac{2}{3(1+\sigma)}, \frac{1}{(32 C^\sharp)^2 g}\right), \\[4pt] 
m  & \geq \frac{8C^*}{g}\left\{ \left|\ln \frac{h^0}{2}\right|+3\sigma \left|\ln \left[(1-\sigma)\frac{h^0}{2}\right]\right| \right\}.
\end{aligned}
\right.
\end{equation}
We emphasize that, the two conditions are fixed sussessively. First $h^0$
is chosen to fulfill the first condition. This fixes the right-hand side of the second inequality  and we might choose a bigger $m$ (which amounts to fix $mh^0$ sufficiently large).  We will take $m$ even larger in what follows (depending on $h^0$ and $\sigma$). We introduce then the sequence of times:
\begin{equation} \label{ode11}
\left\{
\begin{aligned}
t_0       &= \dfrac{1}{4} \sqrt{\dfrac{h^0}{g}} , \\[8pt]
t_{n+1}&  = t_n + \sigma\dfrac{h(t_n)}{2\sqrt{gh^0}},
\end{aligned}
\right.
\end{equation}
and we show 
\begin{Lemma} \label{L01}
Given $\sigma \in (0,1/2)$,  $h^0$ sufficiently small and $mh^0$ sufficiently large, 
the times $t_n$ as computed by \eqref{ode11} are well defined
for all $n \in \mathbb N.$ Furthermore, for any $n \geq 0$ there holds:
\begin{equation}
\label{ode10}
 (1-\sigma)^n \dfrac{h^0}{2}\leq  h(t_n) \leq  \left(1-\frac{\sigma^2}{32}\right)^n \dfrac{3}{2}h^0.
\end{equation}
\end{Lemma}
The proof of this lemma shall end the proof of \cref{T02}. Indeed, 
by relation \eqref{ode10} and definition \eqref{ode11}, we have that the sequence 
of time increments $(t_{n+1}- t_n)_{n\in \mathbb N}$ is dominated by a converging geometric sequence. In particular $t_n$ converges increasingly to a finite time $T_*$
with (applying again \eqref{ode10}) $\lim_{n\to \infty} h(t_n) = 0.$ Since $h$ is at least continuous we get $h(T_*) =0$ preventing from $T>T_*.$ 

\begin{proof}[Proof of Lemma \ref{L01}]
We recall that we assume at first that $h^0$ and $mh^0$ are chosen so that \eqref{ode17bis} holds true. We prove by induction that 
\begin{equation} \tag{$\mathcal P_n$} %\label{induction}
h(t) \in (0,1)\,, \quad \forall \,t \in  [0,t_n]\,, \qquad \text{$h(t_n)$ satisfies \eqref{ode10}}. 
\end{equation}
This entails the expected result.

\medskip

\noindent{\bf Case $n=0.$} By \eqref{ingredient1b} and the choice of $t_0,$ there holds:
\[
\frac{h^0}{2} \leq h(t) \leq \frac{3h^0}{2} \quad  \forall \, t\in [0,t_0]. %
\]
The restriction on $h^0$ in \eqref{ode17bis} implies then that $h(t) \in (0,1)$ for $t \in [0,t_0]$ and  $h^0/2 \leq h(t_0) \leq 3h^0/2.$

\medskip 

\noindent{\bf Induction.} 
Now, fix $n \in \mathbb N$ and assume that ($\mathcal P_k$) holds true for all $k \leq n$.
First, from \eqref{ingredient1b} and \eqref{ode11}, we have:
\[
(1-\sigma)h(t_n) \leq h(t) \leq (1+\sigma)h(t_n) \quad \forall \, t\in [t_n,t_{n+1}].
\]
In particular, there holds:
\begin{equation} \label{ode12}
 (1-\sigma)^{n+1} \dfrac{h^0}{2}  \leq h(t) \leq (1+\sigma)\dfrac{3}{2}h^0 \quad \forall \, t\in [t_n,t_{n+1}].
\end{equation}
By choice of $h^0$ we obtain that $h(t) \in (0,1)$ for $t \in [t_n,t_{n+1}]$ and thus on $[0,t_{n+1}]$ thanks to the induction assumption. We obtain also already the left-hand inequality in \eqref{ode12}:
\[
 (1-\sigma)^{n+1}\dfrac{h^0}{2} \leq  h(t_{n+1})
\]
All that remains concerns the right-hand inequality in \eqref{ode12}. 
For this, we note that $h(t) \leq 1$ on $[t_n,t_{n+1}]$ so that \eqref{ode01b} holds
true. We have then, for $t\in [t_n,t_{n+1}]$,
\begin{align*}
\dot h(t)  & \leq - \dfrac{g}{2} t + C^{\sharp}  gh^0 + \dfrac{C^*}{m} \int_0^t |\ln(h(s))|{\rm d}s  \\
& \leq - \dfrac{g}{2} t_0  +   C^{\sharp} gh^0   + \dfrac{C^*}{m} \int_0^{t_0} |\ln(h(s))|{\rm d}s 
+   \dfrac{C^*}{m} \sum_{k=0}^{n}  \int_{t_k}^{t_{k+1}} |\ln(h(s))|{\rm d}s.
\end{align*}
For the first three terms, we apply that $h(t) \in (h^0/2,1)$ on $(0,t_0),$ 
the definition of $t_0$ and the restrictions on  $h^0$ and $m$ to obtain  that:
\[
-\dfrac{g}{2}t_0 + C^\sharp  gh^0 + \dfrac{C^*}{m} \int_0^{t_0} |\ln(h(s))|{\rm d}s 
\leq
-\dfrac{1}{16} \sqrt{gh^0}+ \dfrac{C^*}{4m}\sqrt{\frac{h^0}{g}}  \left|\ln \frac{h^0}{2}\right|\leq -\dfrac{1}{32} \sqrt{gh^0}.
\]
This entails:
\[
\dot h(t)  \leq  -\dfrac{1}{32} \sqrt{gh^0} 
+   \dfrac{C^*}{m} \sum_{k=0}^{n}  \int_{t_k}^{t_{k+1}} |\ln(h(s))|{\rm d}s.
\]
Introducing again the definition of $(t_{k+1},t_k)$ and the bound below for $h(t)$ on $(t_{k+1},t_k)$
(similar to \eqref{ode12} on $(t_k,t_{k+1})$) we deduce that:
$$
\dot h(t)\leq  -\dfrac{1}{32} \sqrt{gh^0}
+   \dfrac{C^*}{m} \sum_{k=0}^{N}  \sigma\dfrac{h(t_k)}{2\sqrt{gh^0}} \left|\ln \left[(1-\sigma)^{k+1}\frac{h^0}{2}\right] \right|.
$$
Introducing the bound above taken from \eqref{ode10} for $h(t_k)$ in the remaining sum,  we deduce that:
$$
\dot h(t)\leq  -\dfrac{1}{32} \sqrt{gh^0}
+   \dfrac{3C^*\sigma}{4m} \sqrt{\frac{h^0}{g}}\sum_{k=0}^{N}  \left(1-\frac{\sigma^2}{32}\right)^k \left|\ln \left[(1-\sigma)^{k+1}\frac{h^0}{2}\right] \right|.
$$
Since $\sigma \in (0,1/2)$ and the series
$$
\left(\sum_{k=0}^{n}  \left(1-\frac{\sigma^2}{32}\right)^k \left|\ln \left[(1-\sigma)^{k+1}\frac{h^0}{2}\right] \right|\right)_n
$$
is convergent, we may increase the value of $m$ so such that:
\begin{equation}%\label{ode16}
-\dfrac{1}{32} \sqrt{gh^0}
+   \dfrac{3C^*\sigma}{4m} \sqrt{\frac{h^0}{g}}\sum_{k=0}^{n}  \left(1-\frac{\sigma^2}{32}\right)^k \left|\ln \left[(1-\sigma)^{k+1}\frac{h^0}{2}\right] \right|
\leq -\frac{\sigma}{16} \sqrt{gh^0}.
\end{equation}
Integrating this bound above for $\dot{h}(t)$ between $t_{n}$ and 
$t_{n+1}$ we conclude that:
$$
h(t_{n+1})-h(t_n)\leq -\frac{\sigma^2}{32} h(t_n)
$$
This ends up the proof.
\end{proof}

\section{Proof of Lemma \ref{lem_ode}}%\label{sec_de}
The proof of Lemma \ref{lem_ode} is obtained by chosing a suitable test-function
in the weak formulation of \eqref{tre5.0}--\eqref{ci}.  We exhibit now this test-function.  The construction is by now classical (see \cite{MR2354496,MR2481302} among other).  The main point is to define the test-function below the disk. So, given $h \leq 1,$ we set $\mathcal{G}_h$ the subdomain of $\Omega_h$ defined by
\[
\mathcal{G}_h := \left\{ x\in \mathbb{R}^2 \ ; \  |x_1| < \frac 14 \quad x_2 \in (0,H(x_1))\right\},
\]
where $H$ is a graph-parametrization of the bottom part of the disk boundary:
\begin{equation}\label{tre3.0}
H(x_1):=h+\gamma(x_1), \quad \gamma(x_1)=1-\sqrt{1-x_1^2}.
\end{equation}
We also define
\[
\mathcal{G}_h^{1/2} := \left\{ x\in \mathbb{R}^2 \ ; \  |x_1| < \frac 12 \quad x_2 \in (0,H(x_1))\right\}.
\]
We choose $\delta$ sufficiently small so that:
$$
\forall h \leq 1, \quad \left\{(x_1,x_2)\notin \mathcal{G}_h \ ; \  |(x_1,x_2) - (0,1+h)|\in [1+\delta,1+2\delta]\right\}  \subset \Omega.
$$
Then, in $\mathcal{G}_h^{1/2}$, we define
\begin{equation}\label{tre6.7}
a(h,x_1) = \mu_1 h + \mu_2 x_1^2,
\end{equation}
and
\[
\phi_s(x_1,x_2)= x_1 \left[ \left(1- a(h,x_1) \right) \frac{x_2}{H(h,x_1)} +  a(h,x_1) \left(\frac{x_2}{H(h,x_1)}\right)^3 \right],
\]
with $\mu_1=1/6,\mu_2=-3/2.$
This formula corresponds to the stream function of $w_h$ below the disk. We choose to keep abstract letter $\mu_1,\mu_2$
to emphasize from where these explicit values come from in computations. We also consider (with $\delta$ chosen above) $\phi : \mathbb{R}^2 \to \mathbb{R}$ such that
\[
\phi(x_1,x_2)=\left\{
\begin{array}{ll}
x_1 & \text{if} \ |(x_1,x_2)|<1+\delta,\\
0 & \text{if} \ |(x_1,x_2) |\geq 1+2\delta,
\end{array}
\right.
\]
and we set:
$$
{\phi}_{0}(x_1,x_2)=\phi((x_1,x_2) - (0,1+h)).
$$
Finally, our test-function reads in $\Omega_h:$
\begin{equation}\label{mt002}
w_h = \nabla^{\bot} \Psi  = \begin{bmatrix}
-\partial_2 \Psi \\ \
\partial_1 \Psi
\end{bmatrix} 
\quad \text{ where }
\Psi  = 
\left\{
\begin{array}{ll} 
\left(\chi \phi_{0}+(1-\chi){\phi}_0\right) & \text{ in $\mathcal G_{h}$}\\
\phi_{0} & \text{ in $\Omega_h \setminus \mathcal G_h$}
\end{array}
\right.
\end{equation}
where $\chi : \mathbb{R}^2 \to \mathbb{R}$ is a smooth function such that
\begin{equation}\label{mt001}
\chi(x_1,x_2)=\left\{
\begin{array}{ll}
1 & \text{in} \ \left(-\dfrac{1}{4},\dfrac{1}{4}\right)^2,\\
0 & \text{in} \ \mathbb{R}^2\setminus \left[-\dfrac{1}{2},\dfrac{1}{2}\right]^2.
\end{array}
\right.
\end{equation}
Concerning this test-function, we have the following proposition:
\begin{Proposition}\label{P1}
For any $h \leq 1 $, the test-function $w_h$ enjoys the properties: 
\begin{equation}\label{divnul}
 \div w_h =0 \quad \text{in} \ \Omega_h,
\quad 
(w_h-w^*)\cdot n =0 \quad \text{on} \ \partial \Omega_h,
\quad 
2D(w_h)n\cdot \tau  =0  \quad \text{on} \ \partial \Omega, \\
\end{equation}
Moreover, there exists a constant $C$ independent of $h$ such that:
\begin{equation}
\label{maispourquoionumerotelesreferences1?}
 \left\| w_h\right\|_{L^2(\mathcal G_h)} 
+ \left\| \int_0^{x_2} \partial_h w_{h,1} \right\|_{L^2(\mathcal G_h)} 
+ \left\|\int_{0}^{x_2} \nabla w_{h} \right\|_{L^{\infty}(\mathcal G_h)}  
 \leq C ,
\end{equation}
\begin{equation}
\label{tre7.1}
 \|2D(w_h)n \cdot \tau\|_{L^{\infty}(\partial B_h)} + \left\|\int_{0}^{H} \nabla w_{h} \right\|_{L^{\infty}(-1/2,1/2)}   \leq  C
\end{equation}
\begin{equation}
\label{tre7.2}
\dfrac{1}{|\ln(h)|} \int_{\partial \Omega_h}  |[w_h- w^*]\cdot \tau| \ d\gamma  \leq C,
\end{equation}
 \begin{multline}
 \label{tre8.1} h^{\frac 14}\left( \left\| \nabla w_h\right\|_{L^2(\mathcal G_h)}  +
 \|\partial_h w_{h,2}\|_{L^2(\mathcal G_h)} + \|w_h\|_{L^2(\partial B_h \cap \partial \mathcal G_h)} \right.
 \\\left. + \left[ \int_{-1/2}^{1/2}  \left(\int_{0}^{H} \partial_h w_{h,1} \ dz\right)^2 \ dx_1   \right]^{1/2}\right) \leq C,
\end{multline}
and there exists a pressure $q_h$ such that, 
with a constant $C$ independent of $h:$
\begin{equation}\label{tre7.0}
\left\| \Delta w_h - \nabla q_h \right\|_{L^2(\mathcal G_h)} \leq C,
\quad 
h^{\frac 12}\left| e_2 \cdot \int_{\partial B_h} \Sigma(w_h,q_h)n d \gamma \right| \leq C.
\end{equation}
Outside the gap, we have, with a constant $C$ independent of $h:$
\begin{equation}\label{mt003}
\|w_h\|_{L^2(\Omega\setminus \mathcal{G}_h)}
+\|\partial_h w_h\|_{L^2(\Omega\setminus \mathcal{G}_h)}
+\|\nabla w_h\|_{L^\infty(\Omega\setminus \mathcal{G}_h)}
+\left\| \Delta w_h\right\|_{L^2(\Omega\setminus \mathcal{G}_h)}
+\left\| \nabla q_h \right\|_{L^2(\Omega\setminus \mathcal{G}_h)}
 \leq C.
\end{equation}
\end{Proposition}
\begin{proof}
The proof of this lemma is made of long and tedious computations. We recall that they
are based on the explicit formulas for $w_h.$ Beyond these explicit formulas, the main tools are a comparison between powers of $x_2$ and $x_1$ appearing in the numerators
with powers of $H$ appearing on the denominator. For this, we point out that $x_2 \leq H$ while $|x_1 | \leq \sqrt{H}$ in $\mathcal G_h.$ After reduction of formulas based on these comparisons, computing Sobolev norms reduces to the following estimates of integrals: 
\[
\int_{-1/4}^{1/4} \dfrac{|x_1|^{e}}{H(x_1)^{p}}{\rm d}x_1
\leq
\left\{
\begin{array}{ll}
|\ln(h)| & \text{ if $e=1$ and $p=1$} \\
h^{1-p} & \text{ if $e=1$ and $p \geq 2$} \\
h^{1/2-p} & \text{ if $e=0$ and $p \geq 1$} 
\end{array}
\right.
\]
We only provide the computation of the $L^{\infty}$-norm of $2D(w_h)n \cdot \tau$
on $\partial B_h$ to explain the choice of $\mu_1$ and $\mu_2.$ 
Explicit computations show that, on $\partial B_h \cap \partial \mathcal G^{1/2}_h,$
we have:
\[
2D(w)n \cdot \tau  = \dfrac{x_1h(1-6\mu_1)-x_1^3 (3/2 + \mu_2)}{H^2} + O(1).
\]
where $O$ corresponds to a bounded function in $x_1$ independently of $h$. 
We see here that, taking $\mu_1 = 1/6$ and $\mu_2 = -3/2,$ we compensate
the diverging terms and obtain that $2D(w_h)n \cdot \tau$ remains bounded independent of $h.$

\medskip

To conclude the proof, we provide the construction of the pressure $q_h$
which is slightly different from previous computations due to the form of the function $a$.
From the definition of $w_h$, we deduce
\[
\Delta w = \begin{bmatrix} - \partial_{112} \Psi - \partial_{222} \Psi \\
\partial_{111} \Psi + \partial_{122} \Psi  \end{bmatrix}
\quad \text{ in $\Omega_h.$}
\]
We then set ${q}_h = \chi \tilde{q}_h$ in $\mathcal G_h$ ($\chi$ is the truncation function above) that we extend by $0$
and where, for $x \in \mathcal G_h:$
\begin{multline}\label{tre114}
q_h(x)=\int_{-1/2}^{x_1} \dfrac{-6 s a}{H^3} {\rm d}s + \partial_{12} \Psi
	+2 \int_{-1/2}^{x_1} (1-a)  \left( \dfrac{2H'}{H^2} - s\left( 2 \dfrac{(H')^2}{H^3} - \dfrac{H''}{H^2} \right) \right) {\rm d}s \\
	-\frac{ 3x_2^2}{2}\left(  \dfrac{H{''}}{H^2} - 2 \frac{(H')^2}{H^3}\right)  
	+ x_1 \left(\dfrac{3H''H'}{H^3} - 3 \dfrac{(H')^3}{H^4} - \dfrac{H'''}{2 H^2} \right) x_2^2.
\end{multline}
To give the idea of
%explain in little detail 
such a choice, the first integral cancels $\partial_{222} \Psi$
in $\Delta w_{h,1}$ while the second one cancels $\partial_{122} \Psi.$ Doing so, 
we have left a term $-2 \partial_{112} \Psi$ in $\Delta w_{h,1} -\partial_1 q_h$ 
and a term $\partial_{122} \Psi$ in $\Delta w_{h,2} - \partial_2 q_h.$ Unfortunately,
diverging terms remain in these quantities that we compensate with the remaining explicit terms of the pressure.  Indeed, with this choice, we obtain that
\[
|\Delta w_h - \nabla q_h |\leq C\left( 1+\frac{|x_1|}{H} \right) \quad \text{ in $\mathcal G_h$}
\]
which entails the expected result.  Before ending the proof, we also mention that we choose the integrals in $q_h$ starting form $-1/2$ in order to avoid the introduction of diverging term when operating the truncation by $\chi.$ Due to the symmetries of the integrated functions, we could as well have chosen to start from $1/2.$

\medskip

To conclude, standard integration by parts using boundary conditions satisfied by $w_h$ entail that:
\begin{multline*}
e_2 \cdot \int_{\partial B_h} \Sigma(w_h,q_n) n d\gamma 
=  \int_{\partial \Omega_h} w_h\cdot  \Sigma(w_h,q_h)n \ d\gamma +  \int_{\partial \Omega_h} ((w^*-w_h)\cdot \tau)  (2D(w_h)n\cdot \tau) \ d\gamma
\\
=\int_{\Omega_h} w_h\cdot (\Delta w_h-\nabla q_h)  \ dx +2 \int_{\Omega_h} |D(w_h)|^2  \ dx 
+  \int_{\partial B_h} ((w^*-w_h)\cdot \tau)  (2D(w_h)n\cdot \tau) \ d\gamma.
\end{multline*}
Inroducing in this relation the previous results of the proposition we obtain 
the expected bound. This concludes the proof of the proposition.
\end{proof}

Now, we consider a weak solution  $(u,h)$ on $[0,T]$ and we assume that $0 < h \leq 1$. We take the test function $w_h$ obtained in \cref{P1} in the weak formulation \eqref{tre6.1}.  Noticing that we have $\ell=1$ in this case, integrating by parts and introducing the pressure, we obtain: 
\begin{multline}\label{tre6.3b}
\dfrac{\dd}{\dd t} \left[\int_{\Omega_h} \rho_{\mathcal{F}} u \cdot w \ dx + m \dot h \right] 
- \int_{\Omega_h} \rho_{\mathcal{F}} u \cdot  (\partial_t w + (u \cdot \nabla) w  ) \ dx
+ \int_{\Omega_h} u\cdot (-\Delta w + \nabla q) \ dx
\\
+2 \int_{\partial \Omega_h} (u-\lambda w^*)\cdot \tau \  (D(w)n\cdot \tau) \ d\gamma
+\dot h e_2\cdot  \int_{\partial B_h}  \Sigma(w,q)n \ d\gamma
\\
+m_a g
\leq  2\int_{\partial \Omega_h}  |[u+w-(\lambda+1) w^*]\cdot \tau|- |(u-\lambda w^*)\cdot \tau| \ d\gamma.
\end{multline}
We introduce now $\Phi \in C^1((0,1])$ such that
\begin{equation}\label{tre6.4}
\Phi'(h)= e_2\cdot  \int_{\partial B_h}  \Sigma(w_h,q_h)n \ d\gamma.
\end{equation}
and we rewrite \eqref{tre6.3b}:
\begin{multline}\label{tre6.6b}
 \dfrac{\dd}{\dd t} \left[\int_{\Omega_h} \rho_{\mathcal{F}} u \cdot w \ dx
 +m\dot h+\Phi(h)+m_a g t
 \right] - \int_{\Omega_h} \rho_{\mathcal{F}} u \cdot  (\partial_t w + u \cdot \nabla w  ) \ dx
+ \int_{\Omega_h} u\cdot (-\Delta w + \nabla q) \ dx
\\
+2 \int_{\partial \Omega_h} (u-\lambda w^*)\cdot \tau \  (D(w)n\cdot \tau) \ d\gamma
\leq 2 \int_{\partial \Omega_h}  |[w- w^*]\cdot \tau| \ d\gamma.
\end{multline}
For $t$ arbitrary in $(0,T)$, we integrate this last relation and
use the explicit time-dependency of $(w_h,q_h)$ with $\lambda^0 = 0$ to obtain:
\begin{align}
m\dot h(t)+\Phi(h(t))+m_a g t
&  \leq  \left[\int_{\Omega_{h(s)}} \rho_{\mathcal{F}} u(s,x) \cdot w_{h(s)}(x) \ dx\right]_{s=t}^{s=0}
 +\Phi(h^0)  \notag \\
 & + \int_0^t \int_{\Omega_h} \rho_{\mathcal{F}} u \cdot  (\dot{h} \partial_ h w_{h} + u \cdot \nabla w_{h}  ) \ dx \ ds  \notag \\
& + \int_0^t  \int_{\Omega_h} u\cdot (\Delta w_h - \nabla q_h) \ dx \ ds \notag \\
& - 2 \int_0^t  \int_{\partial \Omega_h} (u-\lambda w^*)\cdot \tau \  (D(w_h)n\cdot \tau) \ d\gamma \ ds \notag \\
& + 2 \int_0^t \int_{\partial \Omega_h}  |[w_h- w^*]\cdot \tau| \ d\gamma \ ds. \label{new1.0}
\end{align}
We proceed by estimating the different terms on the right-hand side of this inequality.
Below, we denote by $C_{w}$ any constant depending on our chosen test-functions and on the geometry but independent on $h$.
We extract the dependencies on the other parameters explicitly and thus provide the extensive computations below to conclude the proof.

\medskip

From \eqref{nrj}, \eqref{maispourquoionumerotelesreferences1?} and \eqref{mt003}, we have
\begin{equation}\label{term1}
 \left[\int_{\Omega_h} \rho_{\mathcal{F}} u \cdot w_h \ dx \right]_{s=t}^{s=0} \leq C_{w} \sqrt{\rho_{\mathcal{F}} E^0},
\end{equation}
where we recall that $E^0$ is defined by \eqref{tre4.2}.
To compute the second term, we decompose:
\begin{equation}\label{term2split}
\int_0^t \int_{\Omega_h} \rho_{\mathcal{F}} u \cdot  (\dot{h} \partial_h w + u \cdot \nabla w_h  ) \ dx \ ds
= \int_0^t  \rho_{\mathcal{F}}  \dot{h} \int_{\Omega_h}  \partial_h w_h \cdot u \ dx \ ds
+ 
  \int_0^t  \rho_{\mathcal{F}} \int_{\Omega_h}    u \cdot \nabla w_h   \cdot u  \ dx \ ds.
\end{equation}
To estimate the the first term in the right-hand side of \eqref{term2split}, we split $\Omega_h$ into the domains $\mathcal G_h$ and $\Omega_h\setminus \mathcal{G}_h$. 
%Since $w_h$  is uniformly bounded with its derivative outside $\mathcal G_h$ ( we have,
Combining \eqref{mt003}, \cref{propkorn} and \cref{L02} yields
\begin{align*}
\left| \rho_{\mathcal{F}}  \int_0^t \dot{h} \int_{\Omega_h \setminus \mathcal G_h} u \cdot   \partial_h w  \ dx \ ds\right| 
    & \leq C_w \rho_{\mathcal{F}}  \int_0^t |\dot{h}| \left( \int_{\Omega_h}  |u|^2 \ dx\right)^{\frac 12}  
    \leq C_w \rho_{\mathcal{F}} \int_0^t \int_{\Omega_h} |D(u)|^2 \ dx \ ds\\
   & \leq C_w \rho_{\mathcal{F}} E^0.     
\end{align*}
In the domain $\mathcal G_h,$ we apply \eqref{maispourquoionumerotelesreferences1?}, \eqref{tre7.1}
and \eqref{tre8.1}
\begin{multline*}
\left| \int_{\mathcal G_h}  \rho_{\mathcal{F}} u \cdot  (\dot{h} \partial_h w )  \ dx\right|
\leq 
\rho_{\mathcal{F}} |\dot{h}|  \left[ 
	\left| \int_{\mathcal G_h}   u_1  \partial_{2} \left(\int_0^{x_2} \partial_h w_{h,1} \ dz\right) \ dx  \right| 
	+ \left| \int_{\mathcal G_h}  u_2  \partial_h w_{h,2} \ dx\right|   
\right] 
\\
\leq \rho_{\mathcal{F}} |\dot{h}|  
\left[ 
	\left| \int_{\mathcal G_h}  \partial_2 u_1  \left(\int_{0}^{x_2} \partial_h w_{h,1} \ dz\right) \ dx   \right| 
	+\|u\|_{L^2(\partial B_h)} \left[ \int_{-1/2}^{1/2}  \left(\int_{0}^{H} \partial_h w_{h,1} \ dz\right)^2 \ dx_1   \right]^{1/2} 
	\right.\\ 
	+ \left| \int_{\mathcal G_h}  u_2  \partial_h w_{h,2} \ dx\right|   
\Bigg] 
 \leq C_w \rho_{\mathcal{F}}  |\dot{h}|  \left[ \|\nabla u_1\|_{L^2(\Omega_h)} + \|u\|_{L^2(\partial B_h)} |h|^{-1/4}+ \|u_2\|_{L^2(\Omega_h)} |h|^{-1/4}\right].
\end{multline*}
Consequently, applying \cref{etcacontinueencoreetencore},  \cref{propkorn} and \eqref{nrj}, we conclude that
$$
\left| \int_0^t \int_{\mathcal G_h}  \rho_{\mathcal{F}} u \cdot  (\dot{h} \partial_h w ) \ dx \ ds \right| 
 \leq  C_w  \rho_{\mathcal{F}}  \int_0^t \int_{\Omega_h} |D(u)|^2 \ dx \ ds \leq C_{w} \rho_{\mathcal{F}} E^0.
$$
Outside the gap, we refer to Proposition \ref{propkorn} and \cref{etcacontinueencoreetencore},  \cref{propkorn} and \eqref{nrj} to yield again that:
$$
\left| \int_0^t \int_{\Omega_h \setminus \mathcal G_h}  \rho_{\mathcal{F}} u \cdot  (\dot{h} \partial_h w ) \ dx \ ds \right| 
 \leq  C_w  \rho_{\mathcal{F}}  \int_0^t \int_{\Omega_h} |D(u)|^2 \ dx \ ds \leq C_{w} \rho_{\mathcal{F}} E^0.
$$
We conclude that:
\begin{equation}\label{term2a}
\left| \int_0^t \int_{\Omega_h}  \rho_{\mathcal{F}} u \cdot  (\dot{h} \partial_h w ) \ dx \ ds \right|  \leq C_{w} \rho_{\mathcal{F}} E^0.
\end{equation}

\medskip 

As for the other term  in the right-hand side of \eqref{term2split}, we have again:
\[
 \rho_{\mathcal{F}} \int_{\Omega_h}    (u \cdot \nabla) w_h   \cdot u  \ dx  =
 \rho_{\mathcal{F}} \int_{\mathcal G_h}  (u \cdot \nabla) w_h   \cdot u  \ dx 
  + \rho_{\mathcal{F}} \int_{\Omega_h \setminus \mathcal G_h} (u \cdot \nabla) w_h   \cdot u\ dx.
\]
Using \eqref{mt003} and  \cref{propkorn} entails:
%\[
%\left| \int_0^t \left[\rho_{\mathcal{F}}  \int_{\mathcal G_h} u \cdot \nabla w_h   \cdot u   (1-\chi) \ dx
% + \rho_{\mathcal{F}} \int_{\Omega_h \setminus \mathcal G_h} u \cdot \nabla w_h   \cdot u\ dx\right]\right| \leq C_w \rho_{\mathcal{F}} E^0.
%\]
\[
\left| \int_0^t  \rho_{\mathcal{F}} \int_{\Omega_h \setminus \mathcal G_h} u \cdot \nabla w_h   \cdot u\ dx \ ds\right| \leq C_w \rho_{\mathcal{F}} E^0.
\]
For the other integral, we  integrate by parts:
 \begin{multline*}
 \int_{\mathcal G_h}  u \cdot \nabla w_h \cdot u \ dx 
= \int_{-1/4}^{1/4} (u\otimes u)(x_1,H(x_1)) : \left[\int_0^{H(x_1)} \nabla w_h \ dz\right] {\rm d}x_1 
\\	-  \int_{\mathcal G_h} \partial_2 (u\otimes u) : \left(\int_{0}^{x_2} \nabla w_h \ dz\right) \ dx
\end{multline*}
and we apply here \eqref{maispourquoionumerotelesreferences1?},
\eqref{tre7.1} and  \cref{propkorn} to obtain
 \begin{align*}
\rho_{\mathcal{F}}   \left| \int_0^t \int_{\mathcal G_h}  u \cdot \nabla w_h \cdot u \ dx \ ds\right|
 &  
%\leq  C_w \rho_{\mathcal{F}} \int_0^t \left|  \int_{\partial B_h} |u|^2{\rm d}\sigma  +   \int_{\mathcal G_h}  |u| |\nabla u| \right| 
\leq C_w \rho_{\mathcal{F}} \int_0^t \int_{\Omega_h} |D(u)|^2 \ dx \ ds\\
& \leq C_w \rho_{\mathcal{F}} E^0.  
\end{align*}
We have finally obtained that:
\begin{equation} \label{term3}
\left| \int_0^t \int_{\Omega_h} \rho_{\mathcal{F}} u \cdot  (\dot{h} \partial_h w + u \cdot \nabla w_h  ) \ dx \right|
\leq C_w \rho_{\mathcal{F}} E^0.  
\end{equation}

\medskip

We proceed with the term on the third line of \eqref{new1.0}.
Using \eqref{tre7.0} and \eqref{mt003}, we have:
\begin{align} \notag 
\left| \int_0^t \int_{\Omega_h} u\cdot (-\Delta w_h + \nabla q_h) \ dx \ ds  \right| & 
\leq \left(\int_0^t \|u\|_{L^2(\Omega_h)} \ ds\right) \|-\Delta w_h + \nabla q_h\|_{L^\infty(L^2(\Omega_h))}
\\
& \leq C_{w} \sqrt{t}\|Du\|_{L^2(L^2(\Omega_h))} \leq C_{w}\sqrt{tE^0} \leq C_w E^0 +  t. \label{term4}
\end{align}
Concerning the term on the fourth line of \eqref{new1.0}, from \eqref{tre7.1}, Korn inequality and the energy inequality \eqref{nrj},
we obtain:
\begin{equation}\label{term5}
\left| \int_0^t \int_{\partial \Omega_h} (u-\lambda w^*)\cdot \tau \  (D(w_h)n\cdot \tau) \ d\gamma \ ds  \right| \leq
C_w\int_0^t \int_{\partial \Omega_h} \left| u-\lambda w^* \right| \ d\gamma \ ds  \leq C_w E^0.
\end{equation}
As for the term on the fifth line, recalling \eqref{tre7.2}, we have also:
\begin{equation} \label{nomalacon}
 \int_0^t \int_{\partial \Omega_h}  |[w_h- w^*]\cdot \tau| \ d\sigma \ ds \leq C_w \int_0^t |\ln(h(s))| {\rm d}s.
\end{equation}

\medskip

Combining \eqref{term1}, \eqref{term3}, \eqref{term4}, \eqref{term5}, \eqref{nomalacon} into \eqref{new1.0}, we obtain:
\begin{equation}\label{ode00}
\dot h(t) %+\Phi(h)
\leq  \left( -g + \frac{\rho_{\mathcal{F}}\pi+1}{m}\right){t}+ \frac{C_0}{m}+ \frac{C^*}{m}\int_0^t |\ln(h(s))| \ d s,
\end{equation}
with  $C^* = C_w$ and 
\begin{align*}
C_0 & =C_w(  1 + (\rho_{\mathcal{F}} + 1) E^0) + 2 \max_{(0,1)} |\Phi(h)|
\\
& = C_w\left( 1 + (\rho_{\mathcal{F}} + 1)\left[  \dfrac{\rho_{\mathcal{F}}}{2}\int_{\Omega_{h^0}} |u^0|^2 \ dx + (m-\pi \rho_{\mathcal{F}}) g h^0\right]\right) + 2 \max_{(0,1)} |\Phi(h)|
\end{align*}
Here, we note that $C^*$ is  independent of $m,h^0,$ while choosing
$mh^0 = C_1$ sufficiently large (recall that $\pi \rho_{\mathcal F}/m< 1$), depending on $\rho_{\mathcal F},$ $g$ and $\|u^0\|_{L^2(\Omega_{h^0})},$ we might bound 
\begin{align*}
\dfrac{C_0}{m} & = \left[ \dfrac{C_w}{C_1} \left( 1 + (\rho_{\mathcal{F}} + 1)\left[  \dfrac{\rho_{\mathcal{F}}}{2}\int_{\Omega_{h^0}} |u^0|^2 \ dx \right]\right)    + \dfrac{2 }{C_1}\max_{(0,1)} |\Phi(h)|\right] h^0  + \left(1-\dfrac{\pi \rho_{\mathcal{F}}}{m}\right) g h^0 \\
& \leq 4 gh^0.
\end{align*}
We obtain the expected result with $C^{\sharp}=4.$
This ends the proof of Lemma \ref{lem_ode}.

\bibliographystyle{plain}
\bibliography{reference}

\begin{thebibliography}{10}

\bibitem{BSMT17}
Loredana B\u{a}lilescu, Jorge San~Mart\'{i}n, and Tak\'{e}o Takahashi.
\newblock Fluid-rigid structure interaction system with {C}oulomb's law.
\newblock {\em SIAM J. Math. Anal.}, 49(6):4625--4657, 2017.

\bibitem{Sarka}
Nikolai~V. Chemetov and \v{S}\'{a}rka Ne\v{c}asov\'{a}.
\newblock The motion of the rigid body in the viscous fluid including
  collisions. {G}lobal solvability result.
\newblock {\em Nonlinear Anal. Real World Appl.}, 34:416--445, 2017.

\bibitem{MR1759801}
Carlos Conca, Jorge San Mart\'{\i}n~H., and Marius Tucsnak.
\newblock Existence of solutions for the equations modelling the motion of a
  rigid body in a viscous fluid.
\newblock {\em Comm. Partial Differential Equations}, 25(5-6):1019--1042, 2000.

\bibitem{MR1682663}
B.~Desjardins and M.~J. Esteban.
\newblock Existence of weak solutions for the motion of rigid bodies in a
  viscous fluid.
\newblock {\em Arch. Ration. Mech. Anal.}, 146(1):59--71, 1999.

\bibitem{GVH}
David G\'{e}rard-Varet and Matthieu Hillairet.
\newblock Existence of weak solutions up to collision for viscous fluid-solid
  systems with slip.
\newblock {\em Comm. Pure Appl. Math.}, 67(12):2022--2075, 2014.

\bibitem{GVHW}
David G\'{e}rard-Varet, Matthieu Hillairet, and Chao Wang.
\newblock The influence of boundary conditions on the contact problem in a 3{D}
  {N}avier-{S}tokes flow.
\newblock {\em J. Math. Pures Appl. (9)}, 103(1):1--38, 2015.

\bibitem{MR3466847}
C\'{e}line Grandmont and Matthieu Hillairet.
\newblock Existence of global strong solutions to a beam-fluid interaction
  system.
\newblock {\em Arch. Ration. Mech. Anal.}, 220(3):1283--1333, 2016.

\bibitem{MR1763528}
C\'{e}line Grandmont and Yvon Maday.
\newblock Existence for an unsteady fluid-structure interaction problem.
\newblock {\em M2AN Math. Model. Numer. Anal.}, 34(3):609--636, 2000.

\bibitem{MR1781915}
Max~D. Gunzburger, Hyung-Chun Lee, and Gregory~A. Seregin.
\newblock Global existence of weak solutions for viscous incompressible flows
  around a moving rigid body in three dimensions.
\newblock {\em J. Math. Fluid Mech.}, 2(3):219--266, 2000.

\bibitem{MR2354496}
M.~Hillairet.
\newblock Lack of collision between solid bodies in a 2{D} incompressible
  viscous flow.
\newblock {\em Comm. Partial Differential Equations}, 32(7-9):1345--1371, 2007.

\bibitem{MR2481302}
Matthieu Hillairet and Tak\'{e}o Takahashi.
\newblock Collisions in three-dimensional fluid structure interaction problems.
\newblock {\em SIAM J. Math. Anal.}, 40(6):2451--2477, 2009.

\bibitem{Judakov}
N.~V. Judakov.
\newblock The solvability of the problem of the motion of a rigid body in a
  viscous incompressible fluid.
\newblock {\em Dinamika Splo\v{s}n. Sredy}, (Vyp. 18 Dinamika \v{Z}idkost. so
  Svobod. Granicami):249--253, 255, 1974.

\bibitem{Navier}
Claude Louis Marie~Henri Navier.
\newblock M{\'e}moire sur les lois du mouvement des fluides.
\newblock {\em M{\'e}moires de l’Acad{\'e}mie Royale des Sciences de
  l’Institut de France}, 6(1823):389--440, 1823.

\bibitem{MR1870954}
Jorge~Alonso San~Mart\'{\i}n, Victor Starovoitov, and Marius Tucsnak.
\newblock Global weak solutions for the two-dimensional motion of several rigid
  bodies in an incompressible viscous fluid.
\newblock {\em Arch. Ration. Mech. Anal.}, 161(2):113--147, 2002.

\bibitem{Serre}
Denis Serre.
\newblock Chute libre d'un solide dans un fluide visqueux incompressible.
  {E}xistence.
\newblock {\em Japan J. Appl. Math.}, 4(1):99--110, 1987.

\bibitem{Wang2014}
Chao Wang.
\newblock Strong solutions for the fluid-solid systems in a 2-{D} domain.
\newblock {\em Asymptot. Anal.}, 89(3-4):263--306, 2014.

\end{thebibliography}

\end{document}